\documentclass[11pt,a4paper]{amsart} 
\usepackage{amssymb,amscd,amsmath, young,mathrsfs}
\usepackage{mathrsfs, mathdots} 
\usepackage{float}
\usepackage[usenames]{color}
\usepackage{soul}
\usepackage{array,longtable}
\newcolumntype{C}{>{$}c<{$}}
\newcolumntype{R}{>{$}r<{$}}
\newcolumntype{L}{>{$}l<{$}}
\theoremstyle{plain}
\newtheorem{thm}{Theorem}[section]
\newtheorem{lem}[thm]{Lemma}
\newtheorem{pro}[thm]{Proposition}
\newtheorem{cor}[thm]{Corollary}

\theoremstyle{remark}
\newtheorem{rem}[thm]{Remark}

\newtheorem{dfn}[thm]{Definition}
\newtheorem*{question*}{Question}

\newtheorem*{acknowledgements}{Acknowledgements}

\numberwithin{equation}{section}

\newcommand{\N}{\mathbb{N}}
\newcommand{\Z}{\mathbb{Z}}
\newcommand{\Q}{\mathbb{Q}}

\newcommand{\tensor}{\otimes}

\newcommand{\p}{\mathfrak{p}}

\renewcommand{\epsilon}{\varepsilon}

\renewcommand{\phi}{\varphi}

\def \p {\ensuremath{\mathfrak{p}}}

\makeatletter
\@namedef{subjclassname@2020}{
\textup{2020} Mathematics Subject Classification}
\makeatother

\author{Yifat Moadim-Lesimcha} \address{Department of Mathematics, Bar-Ilan University, Ramat Gan 5290002, Israel}
\author{Michael M.~Schein} \address{Department of Mathematics, Bar-Ilan University, Ramat Gan 5290002, Israel}
\email{yifatshlomov@gmail.com}
\email{mschein@math.biu.ac.il}

\begin{document}

 \title[Pro-isomorphic zeta functions of $D^\ast$ Lie algebras]{Pro-isomorphic zeta functions of some $D^\ast$ Lie lattices of even rank} 

\begin{abstract}
We compute the local pro-isomorphic zeta functions at all but finitely many primes for a certain family of class-two-nilpotent Lie lattices of even rank, parametrized by irreducible monic non-linear polynomials $f(x) \in \Z [x]$.  These Lie lattices correspond to a family of groups introduced by Grunewald and Segal.  The result is expressed in terms of a combinatorially defined family of rational functions. 
\end{abstract}

 \maketitle

\section{Introduction}
\subsection{Pro-isomorphic zeta functions}
Let $G$ be a finitely generated group.  The pro-isomorphic zeta function of $G$, which was originally introduced by Grunewald, Segal, and Smith~\cite{GSS/88}, is the Dirichlet series $\zeta^\wedge_{G}(s) = \sum_{m = 0}^\infty a_m^\wedge (G) m^{-s}$.  Here $s$ is a complex variable and $a_m^\wedge(G)$ is the (necessarily finite) number of subgroups $H \leq G$ of index $m$ such that the profinite completion of $H$ is isomorphic to that of $G$.  In practice it is convenient to interpret this series as counting linear objects.  Let $\mathcal{L}$ be a $\Z$-algebra, which for our purposes is a free $\Z$-module of finite rank endowed with a $\Z$-bilinear multiplication.  Its pro-isomorphic zeta function is the Dirichlet series $\zeta^\wedge_{\mathcal{L}}(s) = \sum_{m = 0}^\infty b^\wedge_m(\mathcal{L}) m^{-s}$, where $b^\wedge_m(\mathcal{L})$ is the number of subalgebras $\mathcal{M} \leq \mathcal{L}$ of index $n$ such that $\mathcal{M} \tensor \Z_p \simeq \mathcal{L} \tensor \Z_p$ for all primes $p$.  An elementary but fundamental result~\cite[Proposition~4]{GSS/88} is the Euler decomposition $\zeta^\wedge_{\mathcal{L}}(s) = \prod_p \zeta^\wedge_{\mathcal{L},p}(s)$, where $\zeta^\wedge_{\mathcal{L},p}(s)$ counts only subalgebras of $p$-power index or, equivalently, $\Z_p$-subalgebras of $\mathcal{L} \tensor \Z_p$ that are isomorphic to $\mathcal{L} \tensor \Z_p$.  An analogous decomposition holds for finitely generated torsion-free nilpotent groups.  If $G$ is such a group, then there is a Lie lattice $\mathcal{L}(G)$, namely a $\Z$-algebra whose multiplication is a Lie bracket, such that $\zeta^\wedge_{\mathcal{L}(G),p}(s) = \zeta^\wedge_{G,p}(s)$ for all but finitely many $p$.  If $G$ is of class two, then this equality holds for all primes $p$; see, for instance,~\cite[\S4]{GSS/88} and~\cite[\S2.1]{BGS/22}.

The present work computes the pro-isomorphic zeta functions of many members of a certain family of class-two-nilpotent Lie lattices of even rank considered by Berman, Klopsch, and Onn~\cite{BKO/even}.   This family corresponds to the representatives constructed by Grunewald and Segal~\cite{GS/84} of commensurability classes of $D^\ast$-groups of even Hirsch length; see Section~\ref{sec:d.star} below.

\subsection{Statement of results}
We now present our main results more precisely.  Let $\Delta(x) \in \Z [x]$ be a primary polynomial, i.e.~$\Delta(x) = f(x)^\ell$ for an irreducible monic polynomial $f(x)$ and $\ell \in \N$.  If $\Delta(x) = x^n + a_{n-1} x^{n-1} + \cdots + a_1 x + a_0$ for $a_i \in \Z$, recall its companion matrix
$$ C_\Delta = \left( \begin{array}{ccccc} 
0 & 1 & 0 & \cdots & 0 \\
0 & 0 & 1 & \cdots & 0 \\
\vdots & \vdots & \vdots & \ddots & \vdots \\
0 & 0 & 0 & \cdots & 1 \\
-a_0 & -a_1 & -a_2 & \cdots & -a_{n-1}
\end{array} \right) \in \mathrm{M}_n(\Z).$$
Let $\mathcal{L}_\Delta$ be the Lie lattice of rank $2n + 2$ with basis 
$x_1, \dots, x_n, y_1, \dots, y_n, z_1, z_2$ and Lie bracket determined by the following:
\begin{itemize}
\item
$[x_i, x_j] = [y_i, y_j] = 0$ for all $1 \leq i,j \leq n$;
\item
$[x_i, y_j] = \delta_{ij} z_1 + (C_\Delta)_{ij} z_2$ for all $1 \leq i,j \leq n$, where $\delta_{ij}$ is the Kronecker delta;
\item
$z_1$ and $z_2$ lie in (and indeed span) the center of $\mathcal{L}_\Delta$.
\end{itemize}

We consider the case where $\Delta(x) = f(x)$ is an {\emph{irreducible}} polynomial of degree $n \geq 2$ and determine $\zeta^\wedge_{\mathcal{L}_f, p}(s)$ for all but finitely many $p$.  Indeed, let $\beta$ be a root of $f(x)$ and consider the number field $K_f = \Q(\beta)$.  Recall that the conductor $\mathcal{F}_f$ is the largest ideal of the ring of integers $\mathcal{O}_{K_f}$ that is contained in $\Z [\beta]$.  For all primes $p$ coprime to $\mathcal{F}_f$, we compute $\zeta^\wedge_{\mathcal{L}_f, p}(s)$ explicitly when $n \geq 3$.  Theorem~\ref{thm:quadratic} treats the case $n = 2$, which was actually treated 35 years ago, under a different name, by Grunewald, Segal, and Smith. 

Moreover, we prove the following finite uniformity statement.
Suppose that $K$ is a number field, $p$ is a prime, and $\mathbf{e} = (e_1, \dots, e_r)$ and $\mathbf{f} = (f_1, \dots, f_r)$ are vectors of natural numbers.  We say that $p$ has decomposition type $(\mathbf{e}, \mathbf{f})$ in $K$ if $p \mathcal{O}_K = \p_1^{e_1} \cdots \p_r^{e_r}$, where the $\p_i \triangleleft \mathcal{O}_K$ are distinct prime ideals with residue fields of cardinality $| \mathcal{O}_K / \p_i | = p^{f_i}$ for every $1 \leq i \leq r$.  This implies that $n = \sum_{i = 1}^r e_i f_i$.  Let $\mathbf{1}$ denote the vector $(1, \dots, 1)$.

\begin{thm} \label{thm:intro}
Let $n \geq 3$, and let $\mathbf{e} = (e_1, \dots, e_r)$ and $\mathbf{f} = (f_1, \dots, f_r)$ satisfy $n = \sum_{i = 1}^r e_i f_i$.  Consider the rational function
\begin{equation*}
W_{\mathbf{e}, \mathbf{f}}(X,Y) = 
\prod_{i = 1}^r \left( \frac{1}{1 - X^{f_i}} \right) \sum_{I \subseteq \{ 1, \dots, r \}} (-1)^{|I|} \frac{X^{\sum_{i \in I} f_i}}{1 - X^{4n + \sum_{i \in I} e_i f_i} Y^{n+2}} \in \Q(X,Y).
\end{equation*}
If $f(x) \in \Z [x]$ is any irreducible monic polynomial of degree $n$, and if the prime $p$ is coprime to $\mathcal{F}_f$ and has decomposition type $(\mathbf{e}, \mathbf{f})$ in $K_f$, then
$\zeta^\wedge_{\mathcal{L}_f, p}(s) = W_{\mathbf{e}, \mathbf{f}}(p, p^{-s})$.

Moreover, if $\mathbf{e} = \mathbf{1}$, i.e. $p$ is unramified in $K_f$, then $W_{\mathbf{1}, \mathbf{f}}(X,Y)$ satisfies the following functional equation:
\begin{equation} \label{equ:functional.eq}
W_{\mathbf{1}, \mathbf{f}}(X^{-1}, Y^{-1}) = (-1)^{r+1} p^{9n - (2n + 4)s} W_{\mathbf{1}, \mathbf{f}}(X,Y).
\end{equation}
\end{thm}
In fact, in~\eqref{equ:wrt.phi} below we realize the functions $W_{\mathbf{1}, \mathbf{f}}$ as specializations of combinatorially defined functions $\Phi_r$ in $2^r$ variables introduced for every $r \in \N$ in Definition~\ref{def:function}.  While these $\Phi_r$ are reminiscent of some functions that have appeared recently in the literature in the context of enumerative problems arising from algebra~\cite{ScheinVoll/15, CSV/19, RV/19, MV/21}, they do not seem to be special cases of them.  Then~\eqref{equ:functional.eq} is immediate from Proposition~\ref{pro:funct.eq}, which proves a self-reciprocity of the functions $\Phi_r$ under inversion of the variables.

We illustrate the explicit formulas of Theorem~\ref{thm:intro} in a simple example:
\begin{cor} \label{thm:cubic}
Let $f(x) = x^3 - 2$.  
\begin{enumerate}
\item
If $p \equiv 1 \, \mathrm{mod} \, 3$ and there exist $a,b \in \Z$ such that $p = a^2 + 27b^2$ (equivalently, if $p$ is totally split in $K_f = \Q(\sqrt[3]{2})$), then
$$ \zeta^\wedge_{\mathcal{L}_f, p}(s) = \frac{1 + 2 p^{13 - 5s} + 2 p^{14 - 5s} + p^{27 - 10s}}{(1 - p^{12 - 5s})(1 - p^{13 - 5s})(1 - p^{14 - 5s})(1 - p^{15 - 5s})}.$$
\item
If $p \equiv 1 \, \mathrm{mod} \, 3$ and there do not exist $a,b \in \Z$ such that $p = a^2 + 27b^2$ (equivalently, if $p$ is inert in $K_f$), then
$$ \zeta^\wedge_{\mathcal{L}_f, p}(s) = \frac{1}{(1 - p^{12 - 5s})(1 - p^{15 - 5s})}.$$
\item
If $p > 2$ and $p \equiv 2 \, \mathrm{mod} \, 3$ (equivalently, if $p \mathcal{O}_{K_f} = \p_1 \p_2$ with $\mathcal{O}_{K_f}/\p_1 \simeq \mathbb{F}_p$ and $\mathcal{O}_{K_f}/\p_2 \simeq \mathbb{F}_{p^2}$) then
$$ \zeta^\wedge_{\mathcal{L}_f, p}(s) = \frac{1 - p^{27 - 10s}}{(1 - p^{12 - 5s})(1 - p^{13 - 5s})(1 - p^{14 - 5s})(1 - p^{15 - 5s})}.$$
\item
If $p \in \{ 2, 3 \}$ (equivalently, if $p$ is totally ramified in $K_f$), then
$$ \zeta^\wedge_{\mathcal{L}_f, p}(s) = \frac{1 + p^{13 - 5s} + p^{14 - 5s}}{(1 - p^{12 - 5s})(1 - p^{15 - 5s})} .$$
\end{enumerate}
\end{cor}
Note that the rational function governing $\zeta^\wedge_{\mathcal{L}_f,p}(s)$ for the ramified primes $p \in \{2, 3 \}$ does not satisfy a functional equation for any symmetry factor.

\begin{rem}  
Observe in passing that the functional equation~\eqref{equ:functional.eq} satisfies~\cite[Conjecture~1.5]{BKO/even}.  Unlike the situation for zeta functions counting subrings, ideals, and some related structures~\cite{Voll/10, LeeVoll/20}, it is not known in general whether local pro-isomorphic zeta functions of nilpotent Lie lattices, even of class two, satisfy functional equations.  See~\cite{BK/15} for an example of a Lie lattice of class four none of whose local pro-isomorphic zeta functions satisfies a functional equation.  However, Berman, Klopsch, and Onn have conjectured, based on a study of known examples, that if $\mathcal{L}$ is graded and $\zeta^{\wedge}_{\mathcal{L},p}(s)$ satisfies a functional equation at almost all primes $p$, then the exponent of $p^{-s}$ in the symmetry factor at almost all primes should be the weight of a minimal grading of $\mathcal{L}$; see~\cite{BKO/even} for definitions and details.  Indeed, the Lie lattices $\mathcal{L}_f$ considered above are naturally graded in the sense of~\cite{BKO/even}, and hence the weight of a minimal grading is $\mathrm{rk}_{\Z} \mathcal{L}_f + \mathrm{rk}_{\Z} [\mathcal{L}_f, \mathcal{L}_f] = (2n + 2) + 2 = 2n + 4$.
\end{rem}

\subsection{The quadratic case} \label{sec:quadratic}
For completeness, we state the pro-isomorphic zeta functions $\zeta^\wedge_{\mathcal{L}_f, p}(s)$ at all but finitely many primes when $f(x) \in \Z [x]$ is an irreducible monic quadratic polynomial.  Note that there are only three decomposition types for a prime in a quadratic number field: inert ($(\mathbf{e}, \mathbf{f}) = ((1),(2))$), totally split ($(\mathbf{e}, \mathbf{f}) = ((1,1), (1,1))$) and totally ramified ($(\mathbf{e}, \mathbf{f}) = ((2),(1))$).  The following claim is essentially due to Grunewald, Segal, and Smith~\cite{GSS/88} and is analogous to Theorem~\ref{thm:intro}.

\begin{thm} \label{thm:quadratic}
Consider the rational function 
$$W(X,Y) = \frac{1}{(1 - X^4 Y^2)(1 - X^5 Y^2)}.$$  
For each of the three decomposition types $(\mathbf{e}, \mathbf{f})$ above, set $W_{\mathbf{e}, \mathbf{f}}(X,Y) = \prod_{i = 1}^r W(X^{f_i}, Y^{f_i})$.
If $f[x] \in \Z [x]$ is an irreducible monic quadratic polynomial and $p$ is coprime to $\mathcal{F}_f$ and has decomposition type $(\mathbf{e}, \mathbf{f})$ in the quadratic number field $K_f$, then 
$$ \zeta^\wedge_{\mathcal{L}_f, p}(s) = W_{\mathbf{e}, \mathbf{f}}(p, p^{-s}).$$
For all three decomposition types, the following functional equation holds:
$$ W_{\mathbf{e}, \mathbf{f}}(X^{-1}, Y^{-1}) = p^{(\sum_{i = 1}^r f_i)(9 - 4s)} W_{\mathbf{e}, \mathbf{f}}(X,Y).$$
\end{thm}
\begin{proof}
Let $\beta_1$ and $\beta_2$ be the roots of $f(x)$.  They are both contained in the Galois extension $K_f / \Q$.  Let $\mathcal{H}$ be the Heisenberg Lie lattice $\mathcal{H} = \langle x, y, z \rangle_{\Z}$ such that $[x,y] = z$ and the product of any other pair of generators vanishes.  Consider $\mathcal{H} \tensor_{\Z} \mathcal{O}_{K_f}$ as a Lie lattice by restriction of scalars.  Since $p$ is coprime to $\mathcal{F}_f$, we have $\Z_p [\beta_1] = \mathcal{O}_{K_f} \tensor_{\Z} \Z_p$ and it is easy to verify that there is an isomorphism
$\varphi: \mathcal{L}_f \tensor_{\Z} \Z_p \stackrel{\sim}{\to} (\mathcal{H} \tensor_{\Z} \mathcal{O}_{K_f}) \tensor_{\Z} \Z_p$ given by
\begin{equation*}
(x_1, x_2, y_1, y_2, z_1, z_2) \mapsto (x \tensor 1, x \tensor \beta_1, y \tensor (-\beta_2), y \tensor 1, z \tensor (-\beta_2), z \tensor 1).
\end{equation*}
Thus $\zeta^\wedge_{\mathcal{L}_f, p}(s) = \zeta^\wedge_{\mathcal{H} \tensor_{\Z} \mathcal{O}_{K_f}, p}(s)$, and the right-hand side of this equality was computed by Grunewald, Segal, and Smith in Theorem~7.1 and Lemma~7.2 of~\cite{GSS/88}; see Theorem~5.10 and Remark~5.12 of~\cite{BGS/22} for an alternative derivation of the same explicit result.
\end{proof}

Observe that the local pro-isomorphic zeta functions $\zeta^\wedge_{\mathcal{L}_f, p}(s)$ appearing in Theorem~\ref{thm:quadratic} decompose as products of factors parametrized by primes of $K_f$ dividing $p$.  This is a special case of a general phenomenon~\cite[Proposition~3.14]{BGS/22}.  The Lie algebras $\mathcal{H} \tensor_{\Z} \Q_p$ satisfy a rigidity property~\cite[Definition~3.8]{BGS/22} originally introduced by Segal~\cite{Segal/89}; as a consequence, the pro-isomorphic zeta function $\zeta^\wedge_{\mathcal{H} \tensor \mathcal{O}_K, p}(s)$ may be computed easily for any number field $K$.  Such rigidity does not hold for the Lie algebras $\mathcal{L}_f \tensor_{\Z} \Q_p$ of Theorem~\ref{thm:intro}; this is essentially a consequence of the arithmetic of the number field $K_f$, which is larger than $\Q$, controlling the local pro-isomorphic zeta functions $\zeta^\wedge_{\mathcal{L}_f, p}(s)$.

\subsection{Overview} \label{sec:overview}
It is a simple but fundamental observation that computations of local factors of pro-isomorphic zeta functions can be reduced to $p$-adic integrals of a certain form.  Consider the $\Q$-Lie algebra $L_\Delta = \mathcal{L}_\Delta \tensor_{\Z} \Q$, and let $\mathbf{G}_{\Delta}$ be its algebraic automorphism group.  This is the algebraic group defined over $\Q$ characterized by the property that $\mathbf{G}_\Delta (E) \simeq \mathrm{Aut}_E (L_\Delta \tensor_{\Q} E)$ for every field $E$ of characteristic zero.  Fixing the ordered basis $(x_1, \dots, x_n, y_1, \dots, y_n, z_1, z_2)$ of $\mathcal{L}_\Delta$ gives an embedding $\mathbf{G}_\Delta \hookrightarrow \mathrm{GL}_{2n + 2}$.  Now set $G_\Delta^+(\Q_p) = \mathbf{G}_\Delta(\Q_p) \cap \mathrm{M}_{2n + 2}(\Z_p)$, and let $G_\Delta(\Z_p) = \mathbf{G}_\Delta(\Q_p) \cap \mathrm{GL}_{2n + 2}(\Z_p)$.  Let $\mu$ be the right Haar measure on the group $\mathbf{G}_{\Delta}(\Q_p)$, normalized so that $\mu(G_\Delta(\Z_p)) = 1$.  Then by~\cite[Proposition~3.4]{GSS/88} we have
\begin{equation} \label{equ:padic.integral}
\zeta^\wedge_{\mathcal{L}_\Delta, p}(s) = \int_{G_\Delta^+(\Q_p)} | \det g |_p^s d \mu,
\end{equation}
where $| \, \cdot \, |_p$ is the normalized valuation on $\Q_p$.  The structure of $\mathbf{G}_\Delta$, for all primary polynomials $\Delta(x) = f(x)^\ell$, was determined by Berman, Klopsch, and Onn; see Proposition~\ref{pro:automorphism.structure} below for the case $\mathrm{deg} \, f(x) \geq 3$.  When $\Delta(x) = f(x)$ is irreducible, the domain of integration of~\eqref{equ:padic.integral} is sufficiently simple that the integral may be computed directly using the Cartan decomposition of $\mathrm{SL}_2(F)$ for $p$-adic fields $F / \Q_p$.  See Remark~\ref{rem:no.cartan} for the reason for the restriction to the irreducible case.  The earlier work cited in the proof of Theorem~\ref{thm:quadratic} also amounts to the computation of an integral~\eqref{equ:padic.integral}.
After establishing several preliminary results, we prove Theorem~\ref{thm:intro} and its corollary in Section~\ref{sec:main.computation} below.

The algebraic group $\mathbf{G}_\Delta$ has a particularly complicated structure when $\Delta(x)$ is a power of a linear polynomial.  The pro-isomorphic zeta functions of $\mathcal{L}_\Delta$ are obtained  in~\cite{BKO/even} for $\Delta(x) = x^2$ and $\Delta(x) = x^3$ after computations substantially more involved than the ones in Section~\ref{sec:main.computation}; it is notable that the simplifying assumptions used in~\cite{duSLubotzky/96} to analyze the integrals~\eqref{equ:padic.integral} do not hold in these cases.

\subsection{Related work and questions} \label{sec:related.work}
This section mentions some results related to our work, as well as directions for future research.

\subsubsection{$D^\ast$-Lie lattices of odd rank} \label{sec:d.star}
A $D^\ast$-group is a radicable, finitely generated, class-two-nilpotent, and torsion free group with finite Hirsch length and with a derived subgroup of Hirsch length two.  Grunewald and Segal~\cite[\S6]{GS/84} classified $D^\ast$-groups up to commensurability.  They showed that every $D^\ast$-group has a central decomposition into indecomposable constituents, which are unique up to isomorphism.  The Lie lattices corresponding to indecomposable $D^\ast$-groups of even Hirsch length are precisely the family $\mathcal{L}_\Delta$, parametrized by primary polynomials $\Delta(x) \in \Z[x]$, that is considered in this article.  The indecomposable $D^\ast$-groups of odd Hirsch length were also determined in~\cite{GS/84}.  The pro-isomorphic zeta functions of the associated Lie lattices, and indeed of a family of Lie lattices generalizing them, were obtained by a lengthy calculation by Berman, Klopsch, and Onn~\cite[Theorem~1.4]{BKO/18}; these have a somewhat different flavor from the functions of Theorem~\ref{thm:intro}.  The results of~\cite{BKO/18} were generalized in~\cite[Theorem~5.17]{BGS/22} to the pro-isomorphic zeta functions of the restriction of scalars to $\Z$ of the base extension of such Lie lattices to the ring of integers of an arbitrary number field.  

\subsubsection{Ideal zeta functions}
We note that the ideal zeta functions $\zeta^{\triangleleft}_{\mathcal{L}}(s)$, namely the Dirichlet series counting ideals of finite index in $\mathcal{L}$, were computed explicitly by Voll~\cite[Propositions~2 and~3]{Voll/04} for Lie lattices $\mathcal{L}$ corresponding to indecomposable $D^\ast$-groups of arbitrary Hirsch length.  He also computed~\cite[Theorem~1.1]{Voll/20} the ideal zeta functions of the generalized family considered in~\cite{BKO/18}.

\subsubsection{General $D^\ast$-groups}
As mentioned above, Lie lattices corresponding to general $D^\ast$-groups arise as central amalgamations of the lattices corresponding to indecomposable $D^\ast$-groups.  The pro-isomorphic zeta function of the central amalgamation of $n$ copies of the Heisenberg Lie lattice $\mathcal{H}$ was computed in~\cite[Theorem~5.10]{BGS/22}; see also~\cite[\S3]{duSLubotzky/96}.  The complexity of the expression obtained grows factorially with $n$, suggesting that computing the pro-isomorphic zeta functions of arbitrary $D^\ast$-groups remains a significant challenge.  By contrast, observe that an algorithm for computing the ideal zeta functions of arbitrary $D^\ast$ groups is given in~\cite[\S3.3]{Voll/04}; see also~\cite[Theorem~1.2]{BS/23}, which shows that ideal zeta functions behave well, in a precise way, under central amalgamation of copies of the same Lie lattice. 

\subsubsection{Uniformity on Frobenius sets}
A well-known consequence of Takagi's existence theorem in global class field theory states that if the Galois closure of a finite extension $K / \Q$ is non-abelian, then the set $\mathrm{Spl}_K$ of totally split primes in $K / \Q$ is not characterized, among the unramified primes, by any finite collection of congruences.  Precisely, there do not exist $m \in \N$ and $S \subseteq \{0, \dots, m-1 \}$ such that $\mathrm{Spl}_K = \{ p: p \equiv a \, \mathrm{mod} \, m, \, a \in S \}$; see, for instance,~\cite[Theorem~7.21]{Garbanati/81}.  In particular, Theorem~\ref{thm:intro} shows that for any monic irreducible $f(x) \in \Z[x]$ such that $K_f / \Q$ has non-abelian Galois closure, the function $p \mapsto \zeta^\wedge_{\mathcal{L}_f,p}(s)$ is not uniform on residue classes.  Corollary~\ref{thm:cubic} provides an example of this phenomenon; note that the Galois closure of $\Q(\sqrt[3]{2}) / \Q$ has Galois group $S_3$.  By contrast, it is follows from the proof of~\cite[Theorem~1.2]{Lagarias/83} that the set of primes of fixed decomposition type in $K_f / \Q$ is a Frobenius set, namely that it is defined by the solvability of a fixed collection of polynomial congruences; see~\cite{Lagarias/83} for precise definitions.  It was recently shown~\cite[Corollary~1.8]{StanojkovskiVoll/21} that the function assigning to a prime $p$ the order of the automorphism group of the group of $\mathbb{F}_p$-points of certain unipotent group schemes is polynomial on Frobenius sets, but not on residue classes.  It would be interesting to describe classes of enumerative problems of algebraic structures whose solution is uniform on Frobenius sets.

\section{Preliminaries}
This section contains two results that will be used in the computation of pro-isomorphic zeta functions and their functional equations that comprise the core of the paper.  We give their proofs here to avoid breaking the flow of the computation later.

\subsection{A combinatorial function}
We introduce a family of combinatorially defined functions in terms of which it will be convenient to express the local pro-isomorphic zeta functions $\zeta^\wedge_{\mathcal{L}_f, p}(s)$.  For every $r \in \N$, let $[r]$ denote the set $\{ 1, 2, \dots, r \}$.
\begin{dfn} \label{def:function}
Let $r \in \N$.  Let $\{ X_I \}_{I \subseteq [r]}$ be a collection of $2^r$ variables, one for each subset $I \subseteq [r]$.  We consider the following function in these variables:
$$ \Phi_r(\{ X_I \}_{I \subseteq [r]}) = \sum_{I \subseteq [r]} (-1)^{|I|} \frac{X_I}{1 - X_I}.$$
\end{dfn}

\begin{pro} \label{pro:funct.eq}
The function $\Phi_r$ satisfies the following self-reciprocity upon inversion of the variables:
$$ \Phi_r(\{ X_I^{-1} \}) = - \Phi_r(\{ X_I \}).$$
\end{pro}
\begin{proof}
Let $\mathcal{P}[r]$ denote the power set of $[r]$.  Writing the rational function over a common denominator, we find that
\begin{multline} \label{equ:functeq.lhs}
  \Phi_r(\{ X_I^{-1} \}) = \\ \sum_{I \in \mathcal{P}[r]} (-1)^{|I|} \frac{X^{-1}_I}{1 - X^{-1}_I} = \sum_{I \in \mathcal{P}[r]} \frac{(-1)^{|I| + 1}}{1 - X_I} = \frac{\sum_{I \in \mathcal{P}[r]} (-1)^{|I|+1} \prod_{J \in \mathcal{P}[r] \atop J \neq I} (1 - X_J) }{\prod_{J \in \mathcal{P}[r]} (1 - X_J)}.
  \end{multline}
Similarly,
\begin{equation} \label{equ:functeq.rhs}
 -\Phi_r(\{ X_I \}) = \frac{\sum_{I \in \mathcal{P}[r]} (-1)^{|I|+1} X_I \prod_{J \in \mathcal{P}[r] \atop J \neq I} (1 - X_J) }{\prod_{J \in \mathcal{P}[r]} (1 - X_J)}.
 \end{equation}
Thus it suffices to show that the numerators of the two expressions are the same.  Multiplying out the parentheses and computing the coefficient of the monomial $\prod_{J \in T} X_J$ for each $T \subseteq \mathcal{P}[r]$, we find that the numerator of~\eqref{equ:functeq.lhs} is
\begin{multline*}
\sum_{I \in \mathcal{P}[r]} (-1)^{|I|+1} \prod_{J \in \mathcal{P}[r] \atop J \neq I} (1 - X_J) = \sum_{I \in \mathcal{P}[r]}(-1)^{|I|+1} \sum_{T^\prime \subseteq \mathcal{P}[r] \atop I \not\in T^\prime} (-1)^{| T^\prime |} \prod_{J \in T^\prime} X_J = \\ \sum_{T \subseteq \mathcal{P}[r]} (-1)^{|T|+1} \left( \sum_{I \in \mathcal{P}[r] \setminus T} (-1)^{|I|} \right) \prod_{J \in T} X_J = \sum_{T \subseteq \mathcal{P}[r]} (-1)^{|T|} \left( \sum_{I \in T} (-1)^{|I|} \right) \prod_{J \in T} X_J,
\end{multline*}
where the last equality follows from the elementary observation that $\sum_{I \in \mathcal{P}[r]} (-1)^{|I|} = \prod_{i = 1}^r (1 - 1) = 0$.  Analogously, the numerator of~\eqref{equ:functeq.rhs} is
\begin{multline*}
\sum_{I \in \mathcal{P}[r]}(-1)^{|I|+1} X_I \sum_{T^\prime \subseteq \mathcal{P}[r] \atop I \not\in T^\prime} (-1)^{|T^\prime |} \prod_{J \in T^\prime} X_J = \sum_{I \in \mathcal{P}[r]} (-1)^{|I| + 1} \sum_{U \subseteq \mathcal{P}[r] \atop I \in U} (-1)^{|U| - 1} \prod_{J \in U} X_J = \\
\sum_{T \subseteq \mathcal{P}[r]} \left( \sum_{I \in T} (-1)^{|T| + |I|} \right) \prod_{J \in T} X_J,
\end{multline*}
where the first equality is obtained by setting $U = T^\prime \cup \{ I \}$.  This completes the proof of our claim.
\end{proof}

\subsection{The Cartan decomposition of $\mathrm{SL}_2(F)$, for a $p$-adic field $F$} \label{sec:haar.measure}
Let $p$ be a prime, and let $v_p$ be the normalized additive valuation on $\Q_p$.  Let $F / \Q_p$ be a finite extension with ring of integers $\mathcal{O}_F$.  Fix a uniformizer $\pi \in \mathcal{O}_F$.  Let $k_F = \mathcal{O}_F / (\pi)$ be the residue field, and let $q$ denote its cardinality.  Given $\lambda \in k_F^\times$, let $[ \lambda ] \in \mathcal{O}_F$ denote the $(q-1)$-st root of unity lifting $\lambda$, and set $[0] = 0$.  Let $I_0 = \{ 0 \}$, and for every $m \in \mathbb{N}$ define the set
$$ I_m = \left\{ [ \lambda_0 ] + \pi [ \lambda_1 ] + \cdots + \pi^{m-1} [ \lambda_{m-1} ] : (\lambda_0, \dots, \lambda_{m-1}) \in k_F^m \right\} \subset \mathcal{O}_F.$$
\begin{lem} \label{lem:coset.reps}
Let $F / \Q_p$ be a finite extension, and let $\pi \in \mathcal{O}_F$ be a uniformizer.  A list of representatives of right cosets of $\mathrm{SL}_2(\mathcal{O}_F)$ in $\mathrm{SL}_2(F)$ is given by
$$
\coprod_{m \geq 0} \left\{ \left( \begin{array}{cc} \pi^m & 0 \\ \pi^{-m} \kappa & \pi^{-m} \end{array} \right) : \kappa \in I_{2m} \right\} \coprod 
\coprod_{m \geq 1} \left\{ \left( \begin{array}{cc} 0 & - \pi^m \\ \pi^{-m} & - \pi^{-m+1} \kappa \end{array} \right) : \kappa \in I_{2m-1} \right\}.
$$
\end{lem}
\begin{proof}
Set $\delta = \left( \begin{array}{cc} \pi & 0 \\ 0 & \pi^{-1} \end{array} \right) \in \mathrm{SL}_2(F)$.  
From the Cartan decomposition
$$ \mathrm{SL}_2(F) = \coprod_{m \geq 0} \mathrm{SL}_2(\mathcal{O}_F) \delta^m \mathrm{SL}_2(\mathcal{O}_F),$$
setting $K_m = \mathrm{SL}_2(\mathcal{O}_F) \cap \delta^{-m} \mathrm{SL}_2(\mathcal{O}_F) \delta^m$ for $m \geq 0$, one deduces the decomposition
$$ \mathrm{SL}_2(F) = \coprod_{m \geq 0} \coprod_{K_m k \in K_m \backslash \mathrm{SL}_2(\mathcal{O}_F)} \mathrm{SL}_2(\mathcal{O}_F) \delta^m k$$
of $\mathrm{SL}_2(F)$ into right cosets of $\mathrm{SL}_2(\mathcal{O}_F)$.  The claim follows by a straightforward computation.  An alternative list of coset representatives may be obtained from~\cite[Proposition~1.1]{Abdellatif/14}, noting that {\emph{left}} cosets of $\mathrm{SL}_2(\mathcal{O}_F)$ correspond to vertices of the Bruhat-Tits tree of $\mathrm{SL}_2(F)$ lying at an even distance from $v_0$, in the notation of~\cite{Abdellatif/14}.
\end{proof}

\begin{cor} \label{cor:measure.computation}
Let $e$ denote the ramification degree of $F / \Q_p$.  Let $\nu_F$ be the right Haar measure on $\mathrm{SL}_2(F)$, with the normalization $\mu(\mathrm{SL}_2(\mathcal{O}_F)) = 1$.  For $a \in \Z_p \setminus \{ 0 \}$, set 
$$S_F(a) = \left\{ \left( \begin{array}{cc} \alpha_{11} & \alpha_{12} \\ \alpha_{21} & \alpha_{22} \end{array} \right) \in \mathrm{SL}_2(F) : \left( \begin{array}{cc} a \alpha_{11} & \alpha_{12} \\ a \alpha_{21} & \alpha_{22} \end{array} \right) \in \mathrm{M}_2(\mathcal{O}_F) \right\}.$$
Then $\nu_F(S_F(a)) = \frac{1 - q^{e v_p(a) + 1}}{1 - q}$.
\end{cor}
\begin{proof}
It is easy to see that $S_F(a)$ is invariant under left multiplication by any element of $\mathrm{SL}_2(\mathcal{O}_F)$ and thus consists of a union of right cosets of $\mathrm{SL}_2(\mathcal{O}_F)$.  Observe that $a \mathcal{O}_F = \pi^{e v_p(a)} \mathcal{O}_F$.
Among the coset representatives listed in Lemma~\ref{lem:coset.reps}, the ones contained in $S_F(a)$ are precisely $\left( \begin{array}{cc} 1 & 0 \\ 0 & 1 \end{array} \right)$ and $\left( \begin{array}{cc} 0 & - \pi^m \\ \pi^{-m} & -\pi^{-m+1} \kappa \end{array} \right)$ for $m \in [ e v_p(a) ]$ and $\kappa = \sum_{i = 1}^{2m-2} \pi^i [\lambda_i ]$ divisible by $\pi^{m-1}$, i.e.~satisfying $\lambda_0 = \cdots = \lambda_{m-2} = 0$.  There are
$$ 1 + \sum_{m = 1}^{e v_p(a)} q^m = \frac{1 - q^{e v_p(a) + 1}}{1 - q}$$
of these.  Since each right coset has measure $1$, the claim follows.
\end{proof}

\section{Computation} \label{sec:main.computation}
\subsection{The algebraic automorphism group}
Let $f(x) \in \Z [x]$ be an irreducible monic polynomial of degree $n \geq 3$.  Consider the primary polynomial $\Delta(x) = f(x)^\ell$ for $\ell \in \N$, and set $K_\Delta = \Q[x] / (\Delta(x))$; this ring has dimension $\ell n$ as a $\Q$-vector space.  To describe the algebraic automorphism group of $L_\Delta$ we define three algebraic subgroups of $\mathbf{GL}_{2\ell n + 2}$.
There is a morphism of algebraic groups
$\rho_2 : \mathrm{Res}_{K_\Delta / \Q} \mathbf{SL}_2 \to \mathbf{SL}_{2\ell n + 2}$ given by
\begin{equation*}
\rho_2 \left( \begin{array}{cc} \alpha_{11} & \alpha_{12} \\ \alpha_{21} & \alpha_{22} \end{array} \right) = \left( \begin{array}{ccc} \iota(\alpha_{11}) & \iota(\alpha_{12}) & \\ \iota(\alpha_{21}) & \iota(\alpha_{22}) & \\ & & I_2 \end{array} \right),
\end{equation*}
where for any $\Q$-algebra $R$ the map $\iota: K_\Delta \tensor_{\Q} R \to \mathrm{M}_{\ell n}(R)$ is determined by $\iota(\beta^i \tensor r) = r C_\Delta^i$ for any $i \in \N \cup \{ 0 \}$ and $r \in R$, and $\beta = x + (\Delta(x)) \in K_\Delta$.  Equivalently, $\iota(\alpha)$ is the matrix of the $R$-linear endomorphism of $K_\Delta \tensor_{\Q} R$ corresponding to multiplication by $\alpha$, with respect to the basis $(\beta^i \tensor 1)_{i = 0}^{\ell n - 1}$.  By a standard exercise in linear algebra, or~\cite[Theorem~1]{KSW/99}, the image of $\rho_2$ is indeed contained in $\mathbf{SL}_{2 \ell n + 2}$.

Consider the embedding of algebraic groups $\rho_1 : \mathbf{G}_m \to \mathbf{GL}_{2 \ell n + 2}$ given by 
\begin{equation*}
 \rho_1(a)  =  \left( \begin{array}{ccc} aI_{\ell n} & 0 & 0 \\ 0 & I_{\ell n} & 0 \\ 0 & 0 & a I_2 \end{array} \right)
 \end{equation*}
 and the embedding $\rho_3 : \mathbf{G}_a^{4 \ell n} \to \mathbf{SL}_{2 \ell n + 2}$ given by
 \begin{equation*}
 \rho_3(c_1, \dots, c_{4\ell n})  =  \left( \begin{array}{ccccc} 1 & &  & c_1 & c_{2 \ell n + 1} \\
 & \ddots & & \vdots & \vdots \\
 & & 1 & c_{2 \ell n} & c_{4 \ell n} \\
 & & & 1 & 0 \\
 & & & 0 & 1 \end{array} \right).
 \end{equation*}

Recall that our chosen $\Z$-basis of $\mathcal{L}_\Delta$ allows us to identify the algebraic automorphism group $\mathbf{G}_\Delta$ of $L_\Delta$ with an algebraic subgroup of $\mathbf{GL}_{2 \ell n + 2}$.  Its structure, which was determined by Berman, Klopsch, and Onn, consists essentially of an internal semi-direct product of the three subgroups of $\mathbf{GL}_{2 \ell n + 2}$ just defined.  There exists a symmetric matrix $\sigma \in \mathrm{GL}_{\ell n}(\Z)$ such that $\sigma C_\Delta \sigma^{-1} = C_\Delta^T$; see~\cite[\S2.1]{BKO/even}.  Set 
$$ \Sigma = \left( \begin{array}{ccc} I_{\ell n} & & \\ & \sigma & \\ & & I_2 \end{array} \right) \in \mathrm{GL}_{2 \ell n + 2} (\Z).$$
\begin{pro} \label{pro:automorphism.structure}
Let $\Delta(x) = f(x)^\ell$, where $f(x) \in \Z[x]$ is an irreducible monic polynomial of degree $n \geq 3$.  Then
$$
\mathbf{G}_\Delta = (\rho_3(\mathbf{G}_a^{4 \ell n}) \rtimes \Sigma (\rho_2(\mathrm{Res}_{K_\Delta / \Q} \mathbf{SL}_2)) \Sigma^{-1} ) \rtimes \rho_1 (\mathbf{G}_m),$$
where the action in each internal semi-direct product is by conjugation.
\end{pro}
\begin{proof}
The subgroup $\mathbf{G}_{0, \Delta} \subset \mathbf{G}_\Delta$ of automorphisms acting trivially on the center $\langle z_1, z_2 \rangle$ is described by~\cite[Theorem~2.3]{BKO/even} and its proof and is $\rho_3 (\mathbf{G}_a^{4 \ell n}) \rtimes \Sigma (\rho_2 (\mathrm{Res}_{K_\Delta / \Q} \mathbf{SL}_2)) \Sigma^{-1}$.  Under the assumption $n \geq 3$, every automorphism acts on the center as a scalar by~\cite[Theorem~1.4]{BKO/evenv2}; this case is not treated in the final version of~\cite{BKO/even}, which focuses on $n = 1$.  It is easy to check that $\rho_1(\mathbf{G}_m) \subset \mathbf{G}_\Delta$.  Thus, for any field extension $E/\Q$, any element of $\mathbf{G}_\Delta(E)$ may be expressed uniquely as a product of an element of $\rho_1(E^\times)$ and one of $\mathbf{G}_{0,\Delta}(E)$, and the claim follows.
\end{proof}

\begin{rem}
Observe that when $f(x) \in \Z[x]$ is an irreducible monic quadratic polynomial, the expression of Theorem~\ref{thm:intro} coincides with the correct local factor $\zeta^\wedge_{\mathcal{L}_f,p}(s)$, as given in Theorem~\ref{thm:quadratic}, when $p$ is inert in $K_f$, but not for the remaining two decomposition types.  We showed in Section~\ref{sec:quadratic} that $\mathcal{L}_f \tensor_{\Z} \Z_p \simeq (\mathcal{H} \tensor_{\Z} \mathcal{O}_{K_f}) \tensor_{\Z} \Z_p$ when $p$ is coprime to $\mathcal{F}_f$, and this gives rise to extra symmetries of $\mathcal{L}_f$.  
Indeed, by~\cite[Theorem~1.4]{BKO/evenv2}, the structure of $\mathbf{G}_f(\Q_p)$ in this case is described by Proposition~\ref{pro:automorphism.structure}, except that $\rho_1(\mathbf{G}_m(\Q_p))$ is replaced by a group isomorphic to $K_f^\times$ rather than $\Q_p^\times$.
\end{rem}

\subsection{Notation} \label{sec:notation}
From now on we assume $\ell = 1$, namely that $\Delta(x) = f(x) \in \Z[x]$ is an irreducible monic polynomial of degree $n \geq 3$.  To simplify the notation, write $\mathbf{G} \subset \mathbf{GL}_{2n + 2}$ for $\mathbf{G}_\Delta$.  Similarly, write $K$ for $K_\Delta$; this is the number field $\Q(\beta)$, where $\beta$ is a root of $f(x)$.  Let $\mathcal{O}_K$ denote the ring of integers of $K$, and recall that the conductor $\mathcal{F}_f$ is the largest ideal of $\mathcal{O}_K$ contained in $\Z [\beta]$.  

Now let $p$ be a rational prime that decomposes in $K$ as $p \mathcal{O}_K = \p_1^{e_1} \cdots \p_r^{e_r}$, where the distinct prime ideals $\p_i \triangleleft \mathcal{O}_K$ have residue fields $\mathcal{O}_K / \p_i$ of cardinality $q_i = p^{f_i}$.  Then $\Q_p \tensor_{\Q} K \simeq F_1 \times \cdots \times F_r$, where for every $i \in [r]$ we write $F_i$ for the localization $K_{\p_i}$.  Similarly, $\Z_p \tensor_{\Z} \mathcal{O}_K \simeq \mathcal{O}_{F_1} \times \cdots \times \mathcal{O}_{F_r}$.

Assume that $p$ is coprime to $\mathcal{F}_f$.  In this case $\Z_p  [x] / (f(x)) = \Z_p \tensor_{\Z} \Z [\beta] = \Z_p \tensor_{\Z} \mathcal{O}_K$.  

\subsection{Setup and evaluation of a $p$-adic integral}
It is immediate from Proposition~\ref{pro:automorphism.structure} that 
\begin{align*}
\mathbf{G}(\Q_p)  = & \rho_1(\Q_p^\times) \ltimes (\Sigma (\rho_2(\mathrm{SL}_2(\Q_p \tensor_{\Q} K))) \Sigma^{-1} \ltimes \rho_3 (\Q_p^{4n})) = \\
 & \rho_1 (\Q_p^\times) \ltimes \left( \Sigma \left( \rho_2 \left( \prod_{i = 1}^r \mathrm{SL}_2(F_i) \right) \right) \Sigma^{-1} \ltimes \rho_3 (\Q_p^{4n}) \right).
\end{align*}
We now explicitly determine the two subsets of $\mathbf{G}(\Q_p)$ necessary for our calculation.
\begin{lem} \label{lem:decompositions}
Suppose that $p$ is coprime to the conductor $\mathcal{F}_f$.
Suppose that $a \in \Q_p^{\times}$, that $A = \left( \begin{array}{cc} \alpha_{11} & \alpha_{12} \\ \alpha_{21} & \alpha_{22} \end{array} \right) \in \mathrm{SL}_2(\Q_p \tensor_{\Q} K)$, and that $\mathbf{c} = (c_1, \dots, c_{4n}) \in \Q_p^{4n}$.  
Then $\rho_3(\mathbf{c}) \Sigma \rho_2(A) \Sigma^{-1} \rho_1(a) \in G(\Z_p)$ if and only if $a \in \Z_p^\times$, whereas $A \in \mathrm{SL}_2(\Z_p \tensor_{\Z} \mathcal{O}_K)$ and $\mathbf{c} \in \Z_p^{4n}$.  Given $a \in \Z_p \setminus \{ 0 \}$, define
\begin{align*}
G_2^+(a) & = & &\left\{ A \in \mathrm{SL}_2(\Q_p \tensor_{\Q} K) :  \left( \begin{array}{cc} a \alpha_{11} & \alpha_{12} \\ a \alpha_{21} & \alpha_{22} \end{array} \right) \in \mathrm{M}_2(\Z_p \tensor_{\Z} \mathcal{O}_K) \right\} \\
G_3^+(a) & = & &\{ \mathbf{c} \in \Q_p^{4n} : (ac_1, \dots, ac_{4n}) \in \Z_p^{4n} \}.
\end{align*}
Then $\rho_3(\mathbf{c}) \Sigma \rho_2(A) \Sigma^{-1} \rho_1(a) \in G^+(\Q_p)$ if and only if $a \in \Z_p \setminus \{ 0 \}$, while $A \in G_2^+(a)$ and $\mathbf{c} \in G_3^+(a)$.
\end{lem}
\begin{proof}
A simple computation shows that
\begin{equation} \label{equ:product.of.matrices}
\rho_3(\mathbf{c}) \Sigma \rho_2(A) \Sigma^{-1} \rho_1(a) = \left( \begin{array}{ccc} 
a \iota(\alpha_{11}) & \iota(\alpha_{12}) \sigma^{-1} & a C_1 \\
a \sigma \iota(\alpha_{21}) & \sigma \iota(\alpha_{22}) \sigma^{-1} & a C_2 \\
0 & 0 & a I_2
\end{array} \right),
\end{equation}
where we use $\rho_1, \rho_2, \rho_3$ to denote the corresponding morphisms on $\Q_p$-points, and where
$$ \begin{array}{lr}
C_1 = \left( \begin{array}{cc} c_1 & c_{2n + 1} \\ \vdots & \vdots \\ c_n & c_{3n} \end{array} \right), \, &
C_2 = \left( \begin{array}{cc} c_{n+1} & c_{3n + 1} \\ \vdots & \vdots \\ c_{2n} & c_{4n} \end{array} \right) .
\end{array}
$$
Observe, given $\alpha \in \Q_p \tensor_{\Q} K$, that $\iota(\alpha) \in \mathrm{M}_n(\Z_p)$ if and only if $\alpha \in \Z_p \tensor_{\Z} \Z [\beta]$, which is equivalent to $\alpha \in \Z_p \tensor_{\Z} \mathcal{O}_K$ by our hypothesis on $p$.
Since $\sigma \in \mathrm{GL}_2(\Z) \subset \mathrm{GL}_2(\Z_p)$, the claim is now immediate from~\eqref{equ:product.of.matrices}.
\end{proof}

The previous claim allows us to express the pro-isomorphic zeta function $\zeta^\wedge_{\mathcal{L}_f, p}(s)$ as an iterated integral.  Indeed, let $\mu_1$ be the right Haar measure on $\Q_p^\times$, normalized so that $\mu_1(\Z_p^\times) = 1$.  Similarly, let $\mu_2$ and $\mu_3$ be the right Haar measures on $\mathrm{SL}_2(\Q_p \tensor_{\Q} K)$ and on $\Q_p^{4n}$, respectively, normalized to $\mu_2(\mathrm{SL}_2(\Z_p \tensor_{\Z} \mathcal{O}_K)) = 1$ and $\mu_3(\Z_p^{4n}) = 1$.  By the first part of Lemma~\ref{lem:decompositions}, these normalizations are compatible with that of the right Haar measure $\mu$ on $\mathbf{G}(\Q_p)$.
Then
\begin{align*} \label{equ:first.step} 
\zeta^\wedge_{\mathcal{L}_f, p}(s) & = & &\int_{G^+(\Q_p)} | \det g |^s_p d \mu (g) = \\ &&
&\int_{\Z_p \setminus \{ 0 \} } \int_{G_2^+(a)} \int_{G_3^+(a)} | \det \rho_3(\mathbf{c}) \Sigma \rho_2(A) \Sigma^{-1} \rho_1(a) |_p^s d \mu_3(\mathbf{c}) d \mu_2 (A) d \mu_1(a) = \\ &&
&\int_{\Z_p \setminus \{ 0 \} } \int_{G_2^+(a)} \int_{G_3^+(a)} |a|_p^{(n + 2)s} d \mu_3(\mathbf{c}) d \mu_2 (A) d \mu_1(a) = \\  &&
&\int_{\Z_p \setminus \{ 0 \} } \int_{G_2^+(a)} |a|_p^{(n + 2)s - 4n} d \mu_2 (A) d \mu_1(a).
\end{align*}
Here the first equality is~\eqref{equ:padic.integral}, the second follows from the second part of Lemma~\ref{lem:decompositions} and~\cite[Proposition~28]{Nachbin/65},  and the last equality holds because the integrand is constant on each set $G_3^+(a)$ and $\mu_3(G_3^+(a)) = |a|_p^{-4n}$ for every $a \in \Z_p \setminus \{ 0 \}$.  Since the integrand is also constant on each $G_2^+(a)$, we have
\begin{equation} \label{equ:first.step}
\zeta^\wedge_{\mathcal{L}_f, p}(s) = \int_{\Z_p \setminus \{ 0 \}} |a|_p^{(n+2)s - 4n} \mu_2(G_2^+(a)) d \mu_1(a).
\end{equation}
Recall the notation defined in Section~\ref{sec:notation}.
\begin{lem} \label{lem:measure.product}
Suppose that $p$ is coprime to $\mathcal{F}_f$.  Then $\mu_2 (G_2^+(a)) = \prod_{i = 1}^r \frac{1 - q_i^{e_i v_p(a) + 1}}{1 - q_i}$ for all $a \in \Z_p \setminus \{ 0 \}$.
\end{lem}
\begin{proof}
The decomposition $\mathrm{SL}_2(\Q_p \tensor_{\Q} K) = \prod_{i = 1}^r \mathrm{SL}_2(F_i)$ induces $\mathrm{SL}_2(\Z_p \tensor_{\Z} \mathcal{O}_K) = \prod_{i = 1}^r \mathrm{SL}_2(\mathcal{O}_{F_i})$ and $G_2^+(a) = \prod_{i = 1} S_{F_i}(a)$, for the sets $S_{F_i}(a)$ 
defined in Section~\ref{sec:haar.measure}, and
the Haar measure $\mu_2$ is the product of the measures $\nu_{F_i}$ defined there.  Hence the claim follows from Corollary~\ref{cor:measure.computation}.
\end{proof}

\begin{rem} \label{rem:no.cartan}
Lemma~\ref{lem:measure.product} is the step in our computation that obliges us to restrict to the case of irreducible $\Delta(x)$.  For a general primary polynomial $\Delta(x)$, it appears to be difficult to compute the measure of the set
$$
\left\{ A \in \mathrm{SL}_2(\Q_p[x]/(\Delta(x))) :  \left( \begin{array}{cc} a \alpha_{11} & \alpha_{12} \\ a \alpha_{21} & \alpha_{22} \end{array} \right) \in \mathrm{M}_2(\Z_p[x] / (\Delta(x))) \right\}
$$
in the absence of a suitable analogue of the $p$-adic Cartan decomposition.
\end{rem}

\subsection{Proof of Theorem~\ref{thm:intro}}  
We can now easily deduce the main result stated in the introduction.
Indeed, let $f(x) \in \Z[x]$ be an irreducible monic polynomial of degree $n \geq 3$.  Let $p$ be a prime coprime to $\mathcal{F}_f$ having decomposition type $(\mathbf{e}, \mathbf{f})$ in the number field $K_f = \Q(x)/(f(x))$.  
We deduce from~\eqref{equ:first.step} and Lemma~\ref{lem:measure.product} that 
$$ \zeta^\wedge_{\mathcal{L}_f, p}(s) = \int_{\Z_p \setminus \{ 0 \} } p^{(4n - (n + 2)s)v_p(a)} \prod_{i = 1}^r \frac{1 - p^{e_i f_i v_p(a) + f_i}}{1 - p^{f_i}} d\mu_1(a).$$
For any $v \geq 0$, we have $\mu_1(\{ a \in \Z_p : v_p(a) = v \}) = \mu_1(p^v \Z_p^\times) = 1$.  Hence
\begin{align*}
\zeta^\wedge_{\mathcal{L}_f, p}(s) & = && \sum_{v = 0}^\infty p^{(4n - (n + 2)s)v} \prod_{i = 1}^r \frac{1 - p^{e_i f_i v + f_i}}{1 - p^{f_i}} = \\ &&& \frac{\sum_{I \subseteq [r]} (-1)^{|I|} \sum_{v = 0}^\infty p^{(4n - (n + 2)s)v} \cdot p^{(\sum_{i \in I} e_i f_i )v + \sum_{i \in I} f_i}}{\prod_{i = 1}^r (1 - p^{f_i})},
\end{align*}
and by summing geometric series we find that $\zeta^\wedge_{\mathcal{L}_f, p}(s) = W_{\mathbf{e}, \mathbf{f}}(p, p^{-s})$ for 
\begin{equation} \label{equ:explicit.function}
W_{\mathbf{e}, \mathbf{f}}(X,Y) = 
\prod_{i = 1}^r \left( \frac{1}{1 - X^{f_i}} \right) \sum_{I \subseteq [r]} (-1)^{|I|} \frac{X^{\sum_{i \in I} f_i}}{1 - X^{4n + \sum_{i \in I} e_i f_i} Y^{n+2}},
\end{equation}
which indeed depends only on $\mathbf{e}$ and $\mathbf{f}$.  If $\mathbf{e} = \mathbf{1}$, we observe by inspection of~\eqref{equ:explicit.function} that
\begin{equation} \label{equ:wrt.phi}
W_{\mathbf{1}, \mathbf{f}}(X,Y) = \frac{1}{X^{4n}Y^{n+2} \prod_{i = 1}^r (1 - X^{f_i})} \Phi_r (\{ X_I \}_{I \subseteq [r]}),
\end{equation}
with $\Phi_r$ as in Definition~\ref{def:function} and $X_I = X^{4n + \sum_{i \in I} f_i} Y^{n+2}$ for all $I \subseteq [r]$.  The claimed functional equation~\eqref{equ:functional.eq} follows from Proposition~\ref{pro:funct.eq} and a simple calculation.

\subsection{Proof of Corollary~\ref{thm:cubic}}
If $f(x) = x^3 - 2$, then $K = K_f = \Q(\sqrt[3]{2})$, and it is a classical fact (see, for instance,~\cite[Theorem~6.4.13]{Cohen/93}) that $\mathcal{O}_K = \Z [\sqrt[3]{2}]$.  Thus $\mathcal{F}_f = (1)$ and Theorem~\ref{thm:intro} applies to all primes.  The discriminant of $K$ is $-108$, so the only ramified primes are $p \in \{ 2, 3 \}$, and one easily verifies that they are both totally ramified.  If $p > 3$, then it follows from~\cite[Corollary~6.4.15]{Cohen/93} and the characterization of the totally split primes in e.g.~\cite[Theorem~9.8]{Cox/13} that $p$ is totally split if $p = a^2 + 27b^2$ (which implies $p \equiv 1 \, \mathrm{mod} \, 3$), whereas $p \mathcal{O}_K = \p_1 \p_2$ with $f_1 = 1$ and $f_2 = 2$ if $p \equiv 2 \, \mathrm{mod} \, 3$ and $p$ is inert otherwise.  With this classification of primes by decomposition type in hand, the claimed formulas are obtained from Theorem~\ref{thm:intro} by straightforward computation.

\begin{acknowledgements}
We are grateful to the anonymous referee for helpful suggestions.
\end{acknowledgements}

\bibliographystyle{amsplain}
\bibliography{evendim.bib}

\providecommand{\bysame}{\leavevmode\hbox to3em{\hrulefill}\thinspace}
\providecommand{\MR}{\relax\ifhmode\unskip\space\fi MR }
\providecommand{\MRhref}[2]{%
  \href{http://www.ams.org/mathscinet-getitem?mr=#1}{#2}
}
\providecommand{\href}[2]{#2}
\begin{thebibliography}{10}

\bibitem{Abdellatif/14}
R.~Abdellatif, \emph{Classification des repr\'{e}sentations modulo {$p$} de
  {${\rm SL}(2,F)$}}, Bull. Soc. Math. France \textbf{142} (2014), 537--589.

\bibitem{BS/23}
T.~Bauer and M.~M. Schein, \emph{Ideal growth in amalgamated powers of
  nilpotent rings of class two and zeta functions of quiver representations},
  Bull. London Math. Soc. \textbf{55} (2023), 1511--1529.

\bibitem{BGS/22}
M.~N. Berman, I.~Glazer, and M.~M. Schein, \emph{Pro-isomorphic zeta functions
  of nilpotent groups and {L}ie rings under base extension}, Trans.\ Amer.\
  Math.\ Soc. \textbf{375} (2022), 1051--1100.

\bibitem{BK/15}
M.~N. Berman and B.~Klopsch, \emph{A nilpotent group without local functional
  equations for pro-isomorphic subgroups}, J. Group Theory \textbf{18} (2015),
  489--510.

\bibitem{BKO/even}
M.~N. Berman, B.~Klopsch, and U.~Onn, \emph{On pro-isomorphic zeta functions of
  ${D}^\ast$-groups of even {H}irsch length}, arXiv/1511.06360, to appear in
  {{Israel J.~Math.}}

\bibitem{BKO/evenv2}
\bysame, \emph{On pro-isomorphic zeta functions of ${D}^\ast$-groups of even
  {H}irsch length}, arXiv/1511.06360v2.

\bibitem{BKO/18}
\bysame, \emph{A family of class-2 nilpotent groups, their automorphisms and
  pro-isomorphic zeta functions}, Math. Z. \textbf{290} (2018), 909--935.

\bibitem{CSV/19}
A.~Carnevale, M.~M. Schein, and C.~Voll, \emph{{Generalized Igusa functions and
  ideal growth of nilpotent Lie rings}}, arXiv:1903.03090, to appear in
  {{Algebra Number Theory}}.

\bibitem{Cohen/93}
H.~Cohen, \emph{A course in computational algebraic number theory}, Graduate
  Texts in Mathematics, vol. 138, Springer-Verlag, Berlin, 1993.

\bibitem{Cox/13}
D.~A. Cox, \emph{Primes of the form {$x^2 + ny^2$}}, 2nd ed., Pure and Applied
  Mathematics, John Wiley \& Sons, Hoboken, 2013.

\bibitem{duSLubotzky/96}
M.~P.~F. du~Sautoy and A.~Lubotzky, \emph{Functional equations and uniformity
  for local zeta functions of nilpotent groups}, Amer. J. Math. \textbf{118}
  (1996), 39--90.

\bibitem{Garbanati/81}
D.~Garbanati, \emph{Class field theory summarized}, Rocky Mountain J. Math.
  \textbf{11} (1981), 195--225.

\bibitem{GS/84}
F.~J. Grunewald and D.~Segal, \emph{Reflections on the classification of
  torsion-free nilpotent groups}, Group Theory. Essays for Philip Hall,
  Academic Press, London, 1984, pp.~121--158.

\bibitem{GSS/88}
F.~J. Grunewald, D.~Segal, and G.~C. Smith, \emph{Subgroups of finite index in
  nilpotent groups}, Invent. Math. \textbf{93} (1988), 185--223.

\bibitem{KSW/99}
I.~Kovacs, D.~S. Silver, and S.~G. Williams, \emph{Determinants of
  commuting-block matrices}, Amer. Math. Monthly \textbf{106} (1999), 950--952.

\bibitem{Lagarias/83}
J.~C. Lagarias, \emph{Sets of primes determined by systems of polynomial
  congruences}, Illinois J. Math. \textbf{27} (1983), 224--239.

\bibitem{LeeVoll/20}
S.~Lee and C.~Voll, \emph{Zeta functions of integral nilpotent quiver
  representations}, Int. Math. Res. Not. IMRN (2023), 3460--3515.

\bibitem{MV/21}
J.~Maglione and C.~Voll, \emph{{Flag {H}ilbert-{H}oincar\'{e} series of
  hyperplane arrangements and their {I}gusa zeta functions}}, arXiv:2103.03640,
  to appear in Israel J.~Math.

\bibitem{Nachbin/65}
L.~Nachbin, \emph{The {H}aar integral}, Van Nostrand, Princeton-Toronto-London,
  1965.

\bibitem{RV/19}
T.~Rossmann and C.~Voll, \emph{{Groups, graphs, and hypergraphs: average sizes
  of kernels of generic matrices with support constraints}}, arXiv:1908.09589,
  to appear in Memoirs Amer. Math. Soc.

\bibitem{ScheinVoll/15}
M.~M. Schein and C.~Voll, \emph{{Normal zeta functions of the Heisenberg groups
  over number rings I -- the unramified case}}, J. London Math. Soc. (2)
  \textbf{91} (2015), 19--46.

\bibitem{Segal/89}
D.~Segal, \emph{On the automorphism groups of certain {L}ie algebras}, Math.
  Proc. Camb. Phil. Soc. \textbf{106} (1989), 67--76.

\bibitem{StanojkovskiVoll/21}
M.~Stanojkovski and C.~Voll, \emph{Hessian matrices, automorphisms of
  {$p$}-groups, and torsion points of elliptic curves}, Math. Ann. \textbf{381}
  (2021), 593--629.

\bibitem{Voll/04}
C.~Voll, \emph{Zeta functions of groups and enumeration in {B}ruhat-{T}its
  buildings}, Amer. J. Math. \textbf{126} (2004), 1005--1032.

\bibitem{Voll/10}
\bysame, \emph{{Functional equations for zeta functions of groups and rings}},
  Ann. of Math. (2) \textbf{172} (2010), 1181--1218.

\bibitem{Voll/20}
\bysame, \emph{Ideal zeta functions associated to a family of class-2-nilpotent
  {L}ie rings}, Q. J. Math. \textbf{71} (2020), 959--980.

\end{thebibliography}
\end{document}